\documentclass{svjour3}                     
\smartqed  
\usepackage[centering, margin={1.2in, 1.4in}, includeheadfoot]{geometry}
\usepackage{tikz,xcolor,hyperref}
\usepackage{amsmath,amsfonts,amssymb}
\usepackage{xcolor}

\newcommand{\C}{\mathbb{C}}
\newcommand{\Hc}{\mathcal{H}}
\newcommand{\K}{\mathcal{K}}
\newcommand{\N}{\mathcal{N}}
\newcommand{\R}{\mathcal{R}}
\newcommand{\B}{\mathcal{B}}
\usepackage{graphicx}
\usepackage{soul}
\usepackage{enumerate}

\begin{document}

\title{GD1 inverse and 1GD inverse for Hilbert space operators
}


\author{Jajati Keshari Sahoo   \and Prdeep Boggarapu \and Ratikanta Behera \and M. Zuhair Nashed
}

\authorrunning{J.K. Sahoo, P. Boggarapu and R. Behera and M. Z. Nashed} 

\institute{ Jajati Keshari Sahoo \at
              Department of Matheamtic, BITS Pilani K.K. Birla Goa Campus, Goa, India.
              \email{jksahoo@goa.bits-pilani.ac.in}   
           \and Pradeep Boggarapu \at
              Department of Matheamtic, BITS Pilani K.K. Birla Goa Campus, Goa, India.              \email{pradeepb@goa.bits-pilani.ac.in} 
                    \and Ratikanta Behera, \at
                    Department of Computational and Data Sciences, Indian Institute of Science, Bangalore, 560012, India, \email{ratikanta@iisc.ac.in}
                    \and M. Zuhair Nashed, \at 
                    Department of Mathematics, University of Central Florida, Orlando, Florida, USA,  \email{M.Nashed@ucf.edu } 
                    }

\date{Received: date / Accepted: date}
 
\maketitle

\begin{abstract}

Mosi\'c and Djordjevi\'c introduced the notation of the gDMP inverse for Hilbert space operators in [J. Spectr. Theory, 8(2):555–573, 2018] by considering generalized Drazin inverse with the Moore-Penrose inverse. This paper introduces two new classes of inverses: GD1 (generalized Drazin and inner) inverse and 1GD (inner and generalized Drazin) inverse for Hilbert space operators. The existence and uniqueness of the GD1 (also 1GD) inverse are discussed along with some properties through core-quasinilpotent decomposition and closed range decomposition operator. We further establish a few explicit representations of the GD1 inverse and their interconnections with generalized Drazin inverse. In addition, we discuss a few properties of GD1 (also 1GD) inverse through binary relation.
\end{abstract}


\section{Introduction}\label{sec1}
\subsection{Background and motivation}
For any complex Hilbert spaces $\mathcal{H}$ and $\K$, we denote $\mathcal{B}(\Hc , \K)$ for the space of all bounded operators from $\Hc$ into $\K$ and $\B(\Hc):=\B(\Hc,\Hc)$. For any operator $T\in \B(\Hc , \K)$, we denote the null space, range space, spectrum  and adjoint of $T$ by $\N(T), \R(T)$, $\sigma(T)$  and  $T^*$ respectively. It is easy to see that for an operator $T\in B(\Hc, \K)$, $\R(T)$ is closed if and only if there exists $X\in \B(\K, \Hc)$ such that $TXT=T$, in this case, we call $X$ is an inner generalized inverse of $T$ and the operator $T$ is relatively regular. We denote the set of inner inverses of $T$ by $T\{1\}$ and element of $T\{1\}$, by $T^{-}$. Further, an element $A$ of a complex Hilbert space $\mathcal{H}$ is called regular (or relatively regular) if there is $X \in \mathcal{H}$. We now recall the Drazin inverse for Hilbert space operators. For an operator $T\in \B(\Hc)$, if there exists an operator $X\in \B(\Hc)$ satisfying the following:
\begin{equation}\label{Di}
XTX=X, \quad TX=XT, \quad   T-T^2X~\text{is nilpotent}  
\end{equation}
then $T$ is said to be Drazin invertible and such an $X$ is called Drazin inverse of $T$ and denoted by $T^D$. The Drazin inverse has various applications such as differential and singular difference equations and Markov chain \cite{Camp79,Chate83,Rako01}.  Because of the non-reflexive condition, the Drazin inverse is very useful in ring theory,  matrix theory (specifically in spectral theory), and various applications of matrix computation. Further, Drazin has discussed the definitions of generalized inverses that give a generalization of the original Drazin inverse in \cite{Drazin92}.  It is known that the Drazin inverse of $T$ exists if and only if $0$ is a pole of the resolvent operator $R_\lambda(T)=(\lambda I-T)^{-1}$ of finite order, say $k$. In this case, $k$ is called the Drazin index of $T$ denoted by ind($T$) and the Drazin inverse of $T$ is unique. If an operator $T\in \B(\Hc)$  has a Drazin index of at most 1, it is called group invertible and $T^D$ is called group inverse of $T$ denoted by $T^\#$. The representations and properties of the Drazin inverses on Hilbert space operators were derived in  \cite{koli,stan1,du,Kol,nas87,wei3}.

Indeed, Moore \cite{BenBook} proposed the concept of generalized inverses of matrices in the 1920s, and Groetsch  in \cite{Groetsch77} generalized the original idea to the bounded linear operators between Hilbert spaces with closed range. 
Further, Nashed \cite{Nashed} discussed the perturbation and approximations of generalized inverses of linear operators  between more general Banach spaces.  It is well known that the perturbation analysis of generalized inverses
in Hilbert and Banach spaces has significantly impacted in practical applications of operator theory (see \cite{Jipu8,Nashed71,Wei01}. Motivated by the work of Mosi\'c and Djordjevi\'c \cite{mosic18}, and the idea of recent works on matrices \cite{GDMP22,mp1,sahoo22}, in this paper, we introduce and study the properties of GD1inverses and
1GD inverses for Hilbert space operators. A brief summarization of the main points of the contribution in this article is listed below.
\begin{itemize}
\item[$\bullet$]   We have introduced the GD1 inverse and its dual (1GD inverse) for Hilbert space operators by extending these inverses as more comprehensive classes of generalized Drazin inverse and an inner inverse. 

\item[$\bullet$]  We have discussed several characterizations of GD1 inverse and 1GD inverse through core-quasinilpotent and closed range decomposition. 

\item[$\bullet$] A few explicit representations of the GD1 inverse and their interconnections with generalized Drazin inverse have been presented explicitly. 

\item[$\bullet$]  A binary relation for both GD1 and 1GD inverse is introduced. Further, we have shown that the relation is a pre-order under some suitable conditions.  
\end{itemize}

\subsection{Outline}
The outline of the paper is as follows. We present some necessary definitions and notation in Sect. 2. Definition, existence, and several explicit representations of the GD1 inverse and their interconnections with  generalized Drazin inverses for Hilbert space operators are considered in Sect. 3. In Sect. 4, we discuss a few properties of binary relations for the GD1 inverse through gDMP partial order. Given the GD1 inverse, we discuss several representations and characterizations of the 1GD inverse in Sect. 5. In addition, we discuss the properties of binary relations for the GD1 through gDMP partial order. The work is concluded along with a few future perspective problems in Sect. 5.
 \section{Preliminaries}
In this {section}, we present a few notations, and  definitions, which will be used in the subsequent sections. 

\subsection{Generalized Drazin Inverse}

An operator $T\in \B(\Hc)$ is said to be quasinilpotent if $I-XT$ is invertible in $\B(\Hc)$ for every $X\in \B(\Hc)$ with $XT=TX$. $T$ is quasinilpotent if and only if $\|T^n\|^{1/n} \to 0$ which is equivalent to $\lambda I -T $ is invertible for all $\lambda \in \C-\{0\}.$ If we replace nilpotent in \eqref{Di} by quasinilpotent, we get the definition of generalized Drazin inverse.

For $T\in \B(\Hc)$, if there exists $X\in \B(\Hc)$ satisfying the following:
\begin{equation}\label{di}
XTX=X, \quad TX=XT, \quad   T-T^2X~\text{is quasinilpotent}  
\end{equation}
then $T$ is said to be Drazin invertible and such an $X$ is called generalized Drazin inverse of $T$ and denoted by $T^d$. Since every nilpotent operator is quasinilpotent, the Drazin inverse is the special case of generalized Drazin inverse.

It is proved that (see \cite{Kol} {Lemma 2.4}), $T\in \B(\Hc)$ is generalized Drazin invertible i.e., $T^d$ exists in $\B(\Hc)$ if and only if there is an idempotent $P \in \B(\Hc)$ commuting with $T$ such that
$$TP ~\text{is quasinilipotent},\quad T+P ~ \text{is invertible}.$$ Here in this case, the generalized Drazin inverse $T^d$ is unique and is given by 
$$ T^d= (T+P)^{-1}(I-P).$$ And then it is proved that the preceding statement is true if and only if $0 \notin acc\;\sigma(T)$.

\subsection{Core-Quasinilpotent Decomposition}
If $T\in \B(\Hc)$ and $0 \notin acc\, \sigma(T)$ then the spectral projector (idempotent) $P$ of $T$ corresponding to $\{0\}$ is given by $P=I-TT^d$. In \cite[Lemma 1.1]{drag}, it is proved that $T$ has the following operator matrix form
\begin{equation}\label{cqnd}
    T=\begin{bmatrix} T_1 ~&~ 0 \\ 0 ~&~ T_2 \end{bmatrix}~:~\begin{bmatrix} \N(P) \\ \R(P) \end{bmatrix} \to \begin{bmatrix} \N(P) \\ \R(P) \end{bmatrix}
\end{equation}
with respect to the decomposition $\Hc=\N(P) \oplus \R(P)$, where $T_1:\N(P)\to \N(P) $ is invertible and $T_2:\R(P) \to \R(P)$ is quasinilpotent. And the generalized Drazin inverse of $T$ is given by 
$$T^d=\begin{bmatrix} T_1^{-1} ~&~ 0 \\ 0 ~&~ 0 \end{bmatrix}~:~\begin{bmatrix} \N(P) \\ \R(P) \end{bmatrix} \to \begin{bmatrix} \N(P) \\ \R(P) \end{bmatrix}.$$

In view of \eqref{cqnd}, if we denote $C_T=\begin{bmatrix} T_1 ~&~ 0 \\ 0 ~&~ 0 \end{bmatrix}$ and $Q_T=\begin{bmatrix} 0 ~&~ 0 \\ 0 ~&~ T_2 \end{bmatrix}$, then we have $T=C_T+Q_T$ and it is known as core-quasinilpotent decomposition of $T$. It can be proved that $C_T=T^2 T^d$, known as the core part of $T$ and $Q_T=TP$, known as the quasinilpotent part of $T$.

\subsection{Closed Range Decomposition}
Let $T\in \B(\Hc)$ be generalized Drazin invertible which has closed range. It is known that (see \cite[Lemma 1.2]{drag} we have the decomposition  $\Hc=\R(T)\oplus^{\perp}\N(T^*)$ with respect to which the operator $T$ has the following matrix representation:
\begin{equation}\label{crd}
    T=\begin{bmatrix} A_1 ~&~ A_2 \\ 0 ~&~ 0 \end{bmatrix}~:~\begin{bmatrix} \R(T) \\ \N(T^*) \end{bmatrix} \to \begin{bmatrix} \R(T) \\ \N(T^*) \end{bmatrix}
\end{equation}
where $D=A_1A_{1}^*+{A}_2A_{2}^*:R(T) \to R(T)$ is positive invertible operator. The generalized Drazin inverse of $T$ with respect to this decomposition is given by (see \cite{mosic18})
$$T^d=\begin{bmatrix} A_1^d ~&~ (A_1^d)^2 A_2 \\ 0 ~&~ 0 \end{bmatrix}~:~\begin{bmatrix} \R(T) \\ \N(T^*) \end{bmatrix} \to \begin{bmatrix} \R(T) \\ \N(T^*) \end{bmatrix}.$$

\section{GD1 inverses}
In this section, we introduce GD1 inverse for Hilbert space operators by combining generalized Drazin inverse and an inner inverse. In addition, we discuss a few characterizations of these inverses along with its interconnection with other generalized inverses.
\subsection{Existence of GD1 inverses}
\begin{proposition}\label{prop-1}
Let $T\in \B(\Hc)$ be generalized Drazin invertible which has closed range.  Then $X=T^dTT^{-}$ is the unique solution of the following conditions:
\begin{equation}\label{eqn-5}
    TX=P_{\R(TT^dT),\N(T^dT^{-})} \mbox{ and } \R(X)\subset \R(TT^d).
\end{equation}
\end{proposition}
\begin{proof}
Let $X=T^dTT^{-}$. Then we verify that  $(TX)^2=TT^dTT^{-}TT^dTT^{-}=TT^dTT^{-}=TX$ and  $\R(X)=\R(T^dTT^{-})\subset \R(TT^d)$. The range condition $\R(TX)=\R(TT^dT)$ is follows from $TT^dT=TT^dTT^{-}T=TXT$. From $T^dT^{-}=(T^d)^2TT^d T T^{-}=(T^d)^2TX$ and $TX=T^2T^dT^{-}$, we obtain $\N(TX)=\N(T^dT^{-})$. Next, we will show the uniqueness of $X$. Suppose there are two operators, $X_1$ and $X_2$, which satisfy \eqref{eqn-5}. Then $T(X_1-X_2)=\bf 0$ and subsequently, $\R(X_1-X_2)\subset \N(T)\subset \N(T^dT)$. Using $\R(X_1)\subset \R(TT^d)$ and $\R(X_2)\subset \R(TT^d)$, we can conclude that $\R(X_1-X_2)\subset \R(TT^d)\cap \N(TT^d)=\{ 0\}$. Hence $X_1=X_2$.
\end{proof}
In view of the Proposition \ref{prop-1}, we define the following representation of GD1 inverse for Hilbert space operators.
\begin{definition}\label{defn-1}
Let $T\in \B(\Hc)$ be generalized Drazin invertible which has closed range and $T^{-}$ be a fixed inner inverse of $T$. An operator $X$ is called the GD1 inverse of $T$ if it satisfies
\begin{center}
   $TX=P_{\R(TT^dT),\N(T^dT^{-})} \mbox{ and } \R(X)\subset \R(TT^d)$. 
\end{center}
\end{definition}
The GD1 inverse of $T$ is denoted by $T^{GD-}$ and it is explicitly represented by $T^{GD-}=T^dTT^{-}$.

One may observe that for every fixed inner $T^{-}$ of $T$, we may get different GD1 inverse of $T$. So, it is important to study the equivalent class of inner inverses, for which we get the same GD1 inverse. For any operator $T\in\B(\Hc)$ with closed range, consider the core-quasinilpotent decomposition of $T$, as given \ref{cqnd}. That is, 
\begin{equation*}
    T=\begin{bmatrix} T_1 ~&~ 0 \\ 0 ~&~ T_2 \end{bmatrix}~:~\begin{bmatrix} \N(P) \\ \R(P) \end{bmatrix} \to \begin{bmatrix} \N(P) \\ \R(P) \end{bmatrix}.
\end{equation*}
Any inner inverse of $T$ will be of the form:
\begin{equation}\label{1i}
    T^{-}=\begin{bmatrix} {T_1}^{-1} ~&~ Y \\ Z ~&~ T_2^- \end{bmatrix}~:~\begin{bmatrix} \N(P) \\ \R(P) \end{bmatrix} \to \begin{bmatrix} \N(P) \\ \R(P) \end{bmatrix}, \mbox{ where $\R(T_2)\subset \N(Y)$, $\R(Z)\subset \N(T_2)$ and $T_2^- \in T_2\{1\}$}.
\end{equation} 
Thus the GD1 inverse of $T$ is given by
\begin{equation}\label{eqqq-7}
    T^{GD-}=\begin{bmatrix} {T_1}^{-1} ~&~ Y \\ 0 ~&~ 0 \end{bmatrix}~:~\begin{bmatrix} \N(P) \\ \R(P) \end{bmatrix} \to \begin{bmatrix} \N(P) \\ \R(P) \end{bmatrix}.
    \end{equation} 
Next, we define a binary relation $\sim_{1}$ on the set $T\{1\}$. For $T^{-},~T^{=}\in T\{1\}$,
\begin{equation*}
    T^{-}\sim_{1}T^{=}\mbox{ if and only if }T^dTT^{-}=T^dTT^{=}.
\end{equation*}
It is clear that $\sim_1$ is an equivalent relation on $T\{1\}$ and the equivalence class of $T^{-}\in T\{1\}$ is given by $[T^{-}]_{\sim_1}=\{T^{=}\in T\{1\}~:~T^dTT^{-}=T^dTT^{=}\}$. We can represent the inner inverse $T^{=}$  of $T$, as
\begin{equation*}
    T^{=}=\begin{bmatrix} {T_1}^{-1} ~&~ Y_1 \\ Z_1 ~&~ T_2^= \end{bmatrix}~:~\begin{bmatrix} \N(P) \\ \R(P) \end{bmatrix} \to \begin{bmatrix} \N(P) \\ \R(P) \end{bmatrix}, \mbox{ where $\R(T_2)\subset \N(Y_1)$, $\R(Z_1)\subset \N(T_2)$ and $T_2^=\in T_2\{1\}$)}. 
\end{equation*}
Hence $T^{=}\in [T^{-}]_{\sim_1}$ if and only if $Y=Y_1$. Further,
\begin{equation*}
    [T^{-}]_{\sim_1}=\left\{\begin{bmatrix} {T_1}^{-1} ~&~ Y \\ Z_1 ~&~ X_1 \end{bmatrix}~:~\begin{bmatrix} \N(P) \\ \R(P) \end{bmatrix} \to \begin{bmatrix} \N(P) \\ \R(P) \end{bmatrix}\in T\{1\}~:~ \R(Z_1)\subset \N(T_2) \mbox{ and } X_1\in T_2\{1\}\right\}.
\end{equation*}
Therefore, $T^{GD-}$ is invariant to the operators of $ [T^{-}]_{\sim_1}$.

Using the the above computations, we state the following result.
\begin{theorem}\label{thm1}
Let $T\in \B(\Hc)$ be generalized Drazin invertible which has closed range and let  \begin{equation*}
    T=\begin{bmatrix} T_1 ~&~ 0 \\ 0 ~&~ T_2 \end{bmatrix}~:~\begin{bmatrix} \N(P) \\ \R(P) \end{bmatrix} \to \begin{bmatrix} \N(P) \\ \R(P) \end{bmatrix}.
\end{equation*} be a core quasinilpotent decomposition of $T$, as defined in \eqref{cqnd}. For an inner inverse of $T$ of the form
\begin{equation*}
    T^{-}=\begin{bmatrix} {T_1}^{-1} ~&~ Y \\ Z ~&~ T_2^- \end{bmatrix}~:~\begin{bmatrix} \N(P) \\ \R(P) \end{bmatrix} \to \begin{bmatrix} \N(P) \\ \R(P) \end{bmatrix}, \mbox{ where $\R(T_2)\subset \N(Y)$ and $\R(Z)\subset \N(T_2)$}, 
    \end{equation*}
the GD1 inverse of $T$ is given by 
\begin{equation}
    T^{GD-}=\begin{bmatrix} {T_1}^{-1} ~&~ Y \\ 0 ~&~ 0 \end{bmatrix}=T^d+\begin{bmatrix} 0 ~&~ Y \\ 0 ~&~ 0 \end{bmatrix}~:~\begin{bmatrix} \N(P) \\ \R(P) \end{bmatrix} \to \begin{bmatrix} \N(P) \\ \R(P) \end{bmatrix}.
    \end{equation} 
Consequently, $T^{GD-}=T^d$ if and only if $Y=0$.
\end{theorem}
In view of the decomposition \eqref{crd}, we can verify the following theorem. 
\begin{theorem}\label{thm-2}
Let $T\in \B(\Hc)$ be generalized Drazin invertible which has closed range and let  \begin{equation*}
     T=\begin{bmatrix} A_1 ~&~ A_2 \\ 0 ~&~ 0 \end{bmatrix}~:~\begin{bmatrix} \R(T) \\ \N(T^*) \end{bmatrix} \to \begin{bmatrix} \R(T) \\ \N(T^*) \end{bmatrix}
\end{equation*} be the decomposition of $T$, as defined in \eqref{crd}.
For an inner inverse 
\begin{center}
    $T^-=\begin{bmatrix} Z_1 ~&~ Z_2 \\ Z_3 ~&~ Z_4 \end{bmatrix}~:~\begin{bmatrix} \R(T) \\ \N(T^*) \end{bmatrix} \to \begin{bmatrix} \R(T) \\ \N(T^*) \end{bmatrix}$ (with $A_1Z_1+A_2Z_3=I$),
\end{center}
the GD1 inverse of $T$ is given by 
$$T^{GD-}=\begin{bmatrix} A_1^d ~&~ A_1^d(A_1Z_2+A_2Z_4) \\ 0 ~&~ 0 \end{bmatrix}~:~\begin{bmatrix} \R(T) \\ \N(T^*) \end{bmatrix} \to \begin{bmatrix} \R(T) \\ \N(T^*) \end{bmatrix},~Z_2\in\B(\N(T^*),\R(T)),~Z_4\in\B(\N(T^*)).$$
Consequently, $T^{GD-}=T^d$ if and only if $A_1^d(A_1Z_2+A_2Z_4)=(A_1^d)^2A_2$.
\end{theorem}

\subsection{Example}

Here we will give an example of operator with closed range on $l^2$, we will find its generalized Drazin inverse and the class of $GD1$ inverses. Consider the standard Schauder basis $\{e_1, e_2, \ldots \}$ for the Hilbert space $l^2$and let $B_1=\{e_1, e_2, e_3\}$ and $B_2= \{e_4, e_5, \ldots \}$, $\Hc_1=span(B_1)$,  $\Hc_2=\overline{span(B_2)}$, and $l^2=\Hc_1\oplus \Hc_2$. Define $T:l^2 \to l^2$ using the block matrix by
\begin{equation}\label{ex1}
    T=\begin{bmatrix} A ~&~ 0 \\ 0 ~&~ D \end{bmatrix}
\end{equation}
with respect the decomposition $l^2=\Hc_1\oplus \Hc_2$, where $A:\Hc_1 \to \Hc_1$ whose matrix representation with respect to the basis $B_1$ is given by $$A=\begin{bmatrix} 1~&~2~&~3\\
0~&~2~&~0\\
0~&~1~&~0\end{bmatrix}$$ and $D:\Hc_2\to \Hc_2$ is the diagonal operator defined by $\displaystyle D e_n=\Big( \frac{2n-1}{n}\Big)e_n$ for $n\geq 4$. Since $D$ is invertible, it can be verified that $$T^d=\begin{bmatrix} A^d ~&~ 0 \\ 0 ~&~ D^{-1} \end{bmatrix}: \begin{bmatrix} \Hc_1 \\ \Hc_2 \end{bmatrix}\to \begin{bmatrix} \Hc_1 \\ \Hc_2 \end{bmatrix}.$$ It is easy to compute that $$A^d=\begin{bmatrix} 1~&~-13/4~&~3\\
0~&~1/2~&~0\\
0~&~1/4~&~0\end{bmatrix}$$ and $$D^{-1}e_n=\Big(\frac{n}{2n-1}\Big) e_n~\text{for}~ n \geq 4.$$ Now will find the class $T\{1\}$. Let $$T^{-}=\begin{bmatrix} X_1 ~&~ X_2 \\ X_3 ~&~ X_4 \end{bmatrix}: \begin{bmatrix} \Hc_1 \\ \Hc_2 \end{bmatrix}\to \begin{bmatrix} \Hc_1 \\ \Hc_2 \end{bmatrix}$$
In view of \eqref{ex1} and $T T^{-} T=T$ with easy computations we get that $X_1\in A\{1\},~AX_2=0,~X_3A=0$ and $X_4=D^{-1}$. Thus
$$T^{GD-}=T^d T T^-=\begin{bmatrix} A^d A A^{-} ~&~ 0 \\ X_3 ~&~ D^{-1} \end{bmatrix}: \begin{bmatrix} \Hc_1 \\ \Hc_2 \end{bmatrix}\to \begin{bmatrix} \Hc_1 \\ \Hc_2 \end{bmatrix}.$$  Hence the class of GD1 inverses of $T$ is parameterized by $A^{-}\in A\{1\}$ and $X_3\in \B(\Hc_1, \Hc_2)$ with $X_3A=0.$

\subsection{Characterizations of GD1 inverses}
The GD1 inverse can be obtained from outer inverse ($X$ is called an outer inverse of $T$ if $XTX=X$) with a prescribed range and null space, as  presented below.
 \begin{theorem}\label{thm2}
Let $T\in \B(\Hc)$ be generalized Drazin invertible which has closed range. Then 
\begin{enumerate}[(i)]
    \item $T^{GD-}T=T^dT$ is a projector onto $\R(T^dT)$ along $\N(T^dT)$.
    \item $TT^{GD-}$ is a projector onto $\R(TT^dT)$ along $\N(T^dT^-)$.
    \item  $T^{GD-}TT^{GD-}=T^{GD-}$, $\R(T^{GD-})=\R(T^dT)$ and $\N(T^{GD-})=\N(T^dT^{-})$.
\end{enumerate}
\end{theorem}
\begin{proof}
(i) It is trivial.\\
(ii) Using the representation of $T^{GD-}$, we obtain
\begin{equation}\label{eqn-9}
(TT^{GD-})^2= TT^dTT^-TT^dTT^-=TT^dTT^dTT^-=TT^dTT^-=TT^{GD-}
\end{equation}
and
\begin{equation}\label{eqq-10}
TT^dT=T^dTT^-T=T^{GD-}T.
\end{equation}
Using \eqref{eqn-9} and  \eqref{eqq-10}, we get $\R(TT^{GD-})=\R(TT^dT)$. Since $TT^{GD-}=T^2T^dT^-$ and $T^dT^-=T^dT^{GD-}$, we have $\N(TT^{GD-})=\N(T^dT^-)$.
\item[(iii)]
Clearly $T^{GD-}TT^{GD-}=T^{GD-}$. From 
$T^{GD-} = T^dTT^-$, $T^dT=T^{GD-}T$, and $T^dT^-=T^dT^{GD-}$, we obtain $\R(T^{GD-})=\R(T^dT)$ and 
$\N(T^{GD-}) = \N(T^dT^-)$.
\qed
\end{proof}
Notice that if $TT^{GD-}=T^{GD-}T$ then $T^{GD-}=T^dTT^{GD-}=T^dT^{GD-}T=T^d$. Conversely, if $T^{GD-}=T^d$, then we can easily verify that $TT^{GD-}=T^{GD-}T$. Thus we state the following result.
\begin{theorem}\label{thm3}
Let $T\in \B(\Hc)$ be generalized Drazin invertible which has closed range. Then $TT^{GD-}=T^{GD-}T$ if and only if $T^{GD-}=T^d$.
\end{theorem}

\begin{theorem}\label{thm4}
Let $T\in \B(\Hc)$ be generalized Drazin invertible which has closed range. Then $X=T^{d}T T^-$ is the unique solution of operator equations: \begin{equation}\label{Ope}
XTX=X, ~~ T^dX =T^d T^-,~~\text{and}~~XT=T^d T.
\end{equation}
\end{theorem}
\begin{proof}
In view of Theorem \ref{thm2} (i) and (iii), $X=T^d T T^-$ satisfies $XT=T^d T$ and $XTX=X$. For $X=T^d T T^-$, we have that $T^d X=(T^d)^2 T T^-=T^d T^-$. Therefore $X=T^d T T^-$ satisfies the operator equations \eqref{Ope}.

Now we will prove the uniqueness. Suppose there exists $X_1$ and $X_2$ which satisfies the operator equations \eqref{Ope}. Then
\begin{align*}
X_1=X_1 T X_1=T^d T X_1=T T^d T^-=T T^d X_2=X_2 T X_2=X_2.    \end{align*}
\end{proof}

The idempotent property of GD1 inverse is discussed in the below result.
\begin{proposition}\label{p3}
 Let $T\in \B(\Hc)$ be generalized Drazin invertible which has closed range.  Then
 \begin{itemize}
     \item[(i)] $(T^{GD-})^2=T^{d}T^{-}$.
     \item[(ii)] $T^{GD-}$ is idempotent  if and only if $T^{GD-}=T^{d}T^{-}$ if and only if $T^{GD-}=TT^{GD-}$.
 \end{itemize}
  \end{proposition}
\begin{proof}
(i) $(T^{GD-})^2=T^dTT^-TT^dT^-=T^dTT^dT^-=T^{d}T^{-}$.\\
(ii) Let $T^{GD-}$ be idempotent. Then by part (i), we have 
\begin{center}
  $T^{GD-}=T^dT-$ and $T^{GD-}=TT^dT^-=TT^{GD-}$.  
\end{center}
Conversely, if $T^{GD-}=T^dT-$, then $(T^{GD-})^2=T^dT^-=(T^{GD-})^2$. Further if $T^{GD-}=TT^{GD-}$, then 
\begin{center}
    $T^{GD-}=T (T^d T T^-)=T (T^d T T^d)(T T^-)=T T^d (T T^- T)T^d T T^- =(T T^{GD-}) (T T^{GD-})=(T^{GD-})^2$.
\end{center}
\end{proof}

\begin{theorem}
Let $T\in \B(\Hc)$ be generalized Drazin invertible which has closed range.  Then
\begin{enumerate}[(i)]
    \item $T T^{GD-}=T T^-$ if and only if $T T^d T=T$.
    \item $T T^{GD-}=T T^d $ if and only if $T^{GD-}=T^d$.
    \item $(T^{GD-})^m = (T^d)^{m-1} T^-$ for $m\geq 2$.
\end{enumerate}
\end{theorem}
\begin{proof}
(i) Let $T T^{GD-}=T T^-$. Then 
\begin{center}
   $TT^dT=TT^dTT^-T=TT^{GD-}T=TT^-T=T$. 
\end{center}
 Conversely, let $T T^d T=T$. Post-multiplying by $T^-$, we obtain $T T^{GD-}= T T^-$.\\
 (ii) Let $TT^{GD-}=TT^d$. Then $T^{GD-}=T^dTT^{GD-}=T^dTT^d=T^d$. The converse part is trivial.\\
 (iii) We prove this identity by induction on $m$. The identity is true for $m$ since 
 \begin{center}
   $(T^{GD-})^2=(T^d T T^-)(T^d T T^-)=T^d(T T^- T) T^d T^-=T^d T^- $.   
 \end{center}
 Now assume the identity is true for $m=k$, that is , $(T^{GD-})^k= (T^d)^{k-1}T^-$. Now 
\begin{align*}
    (T^{GD-})^{k+1}= (T^{GD-})^{k}TT^d T^-
    = (T^d)^{k-1}T^-TT^d T^-=(T^d)^kTT^-TT^dT^-=(T^d)^kT^-.
\end{align*}
\end{proof}

\begin{theorem}
Let $T\in \B(\Hc)$ be generalized Drazin invertible which has closed range. $T^{GD-}$ is idempotent if and only if $T^{GD-}=T^n (T^{GD-})^m$ for any $n\in \mathbb{N}\cup \{0\}$ and $m\in \mathbb{N}$.
\end{theorem}
\begin{proof}
Suppose that $T^{GD-}$ is idempotent. Let $n\in \mathbb{N}\cup \{0\}$ and $m\in \mathbb{N}$. In view of Proposition \ref{p3} (ii), we have $T^n (T^{GD-})^m=T^n T^{DG-}=T^{DG-}$. Converse follows from Proposition \ref{p3} by taking $n=1$ and $m=1.$
\end{proof}
\begin{remark}
If $T$ is idempotent element then then $T$ is generalized Drazin invertible and $T^d=T$. Indeed, if $T$ is idempotent then for $X=T$ we have $XTX=X, XT=TX$ and $T-T^2X=0$ is quasinilpotent. 
\end{remark}
Next we will prove that GD1 inverse of $T$ is the Generalized Drazin inverse of $T^2 T^{-1}$. In order to that we need the following proposition:
\begin{proposition}\label{prop-4}
If $T\in \B(\Hc)$ is generalized Drazin invertible with closed range and $X$ is any solution of operator equations:
\begin{equation}\label{Ope2}
    XTX=X~~\text{and}~~ XT=T^d T,
\end{equation}
then $T^2(X-T^-)$ is quasinilpotent. In particular $(T^dT-I)T^2 T^-$ is quasinilpotent.
\end{proposition}
\begin{proof}
Since a bounded operator is quasinipotent if and only if its spectrum is $\{0\}$, it is enough to prove that $\sigma(T^2(X-T^-))\cup\{0\}=\{0\}$.
$$\sigma(T^2(X-T^-))\cup\{0\}=\sigma(T(X-T^-)T)\cup\{0\}=\sigma(TXT-T)\cup\{0\}=\sigma(T^2T^d-T))\cup\{0\}=\{0\},$$ since $T^2 T^d-T$ is quasinilpotent.

We know that $X=T^d T T^-$ satisfies operator equations \eqref{Ope2} hence $T^2(T^dTT^--T^-)=(T^dT-I)T^2 T^-$ is quasinipotent.
\end{proof}

\begin{theorem}\label{thm7}
Let $T\in \B(\Hc)$ be generalized Drazin invertible which has closed range. Then $T^2T^-$ is generalized Drazin invertible and $(T^2 T^-)^d=T^{GD-}$.
\end{theorem}
\begin{proof}
We need to claim that
$
T^{GD-}(T^2T^-) T^{GD-}=T^{GD-},~ (T^2T^-) T^{GD-}=T^{GD-}(T^2T^-),$
and $T^2T^- -(T^2T^-)^2 T^{GD-}$ is quasinilpotent. Using the definition of $T^{GD-}$, we verify that 
$$T^{GD-}(T^2T^-) T^{GD-}=T^d T T^-(T^2T^-)T^d T T^-=(T^d)^2 T^2 T^-=T^{GD-},$$
and
$$(T^2T^-) T^{GD-}=(T^2 T^-) (T^d T T^-)=T^2 T^d T^-=T^d T^2 T^-=(T^d T T^-)(T^2 T^-)=T^{GD-}(T^2 T^- ).$$ 
In view of Proposition \ref{prop-4},  
$$ T^2T^--(T^2T^-)^2 T^{GD-}=T^2T^--(T^2T^-)(T^2 T^-) T^d T T^-=T^2T^--T^3 T^d T^-=(I-T^d T) (T^2 T^-)$$ is quasinilpotent.
\end{proof}

\begin{remark}
We can also prove Theorem \ref{thm7} using the Cline’s formula of generalized Drazin invertibility (see \cite{Jchen}, Theorem 2.1), which is stated below.
\begin{lemma}\label{l1}
Let $T\in \B(\Hc, \K)$ and $S\in \B(\K, \Hc)$. If $TS$ is generalized Drazin invertible in $\B(\K)$, then so is $ST$ in $\B(\Hc)$ and 
$$(ST)^d=S ((TS)^d)^2 T. $$
\end{lemma}

\end{remark}

\begin{theorem}\label{g1ep}
Let $T\in \B(\Hc)$ be generalized Drazin invertible which has closed range. Then $$\lim_{n\to \infty}(T^-T^{n+1}T^d-T^{n+1}T^dT^-)=0$$ if and only if $\displaystyle\lim_{n\to \infty}(T^-T^{n+2}T^d-T^{n+1}T^d)=0$ and $\displaystyle \lim_{n\to \infty}(T^{n+2}T^dT^--T^{n+1}T^d)=0$.
\end{theorem}
\begin{proof}
Let $\lim_{n\to \infty}(T^-T^{n+1}T^d-T^{n+1}T^dT^-)=0$. Then 
\begin{eqnarray*}
\|T^-T^{n+2}T^d-T^{n+1}T^d\|&=&\|T^-T^{n+2}T^d-T^{n}T^dT\|=\|T^-T^{n+1}T^dT-T^{n}T^dTT^-T\|\\
&\leq&\|T^-T^{n+1}T^d-T^{n+1}T^dT^-\|\|T\|\to 0 \mbox{ as } n\to \infty
\end{eqnarray*}
and 
\begin{equation*}
\|T^{n+2}T^dT^--T^{n+1}T^d\|\leq\|T\|\|T^-T^{n+1}T^d-T^{n+1}T^dT^-\|\to 0 \mbox{ as } n\to \infty.
\end{equation*}
conversely,
\begin{eqnarray*}
\|T^-T^{n+2}T^d-T^{n+2}T^dT^-\|&=&=\|T^-T^{n+2}T^d-T^{n+1}T^d+T^{n+1}T^d-T^{n+2}T^dT^-\|\\
&\leq& \|T^-T^{n+2}T^d-T^{n+1}T^d\|+\|T^{n+1}T^d-T^{n+2}T^dT^-\|\\
&\to&0 \mbox{ as }n\to\infty.
\end{eqnarray*}
\end{proof}

\section{A binary relation based on GD1 inverse}
In view of gDMP pre-order \cite{mosic18}, we introduce the following binary relation for GD1 inverse.

\begin{definition}\label{gd1rel}
Let $T\in \B(\Hc)$ and $S\in \B(\Hc)$ be generalized Drazin invertible which has closed range. We will say that  $S$ is below $T$  under the relation  $\leq ^ {GD-}$ if $SS^{GD-}= TS^{GD-}$ and $S^{GD-}S = S^{GD-}T$. We denote such relation by $S \leq ^{GD-}T$.
\end{definition}

\begin{proposition}\label{pp-5}
Let $T\in \B(\Hc)$ and $S\in \B(\Hc)$ be generalized Drazin invertible which have closed range. Then
\begin{enumerate}
    \item [(i)] $SS^{GD-}= TS^{GD-}$  if and only if $SS^d = TS^d$.
    \item [(ii)] $S^{GD-}S = S^{GD-}T$ if and only if $S^d = S^dS^{-}T$.
\end{enumerate}
\end{proposition}
\begin{proof}
(i) Let $SS^{GD-}= TS^{GD-}$. Then 
\begin{equation*}
    SS^d=SS^{GD-}SS^d=TS^{GD-}SS^d=TS^dSS^-SS^d=TS^dSS^d=TS^d.
\end{equation*}
Conversely, let $SS^d=TS^d$. Then $SS^{GD-}=SS^dSS^-=TS^dSS^-=TS^{GD-}$.\\
(ii) Let $S^{GD-}S = S^{GD-}T$. Then $S^d=S^dSS^d=S^dS^{GD-}S=S^dS^{GD-}T=S^dS^-T$. Conversely, if $S^d=S^dS^-T$ then $S^{GD-}T=SS^dS^-T=SS^d=S^{GD-}S$.
\end{proof}
In view of Definition \ref{gd1rel} and Proposition \ref{pp-5}, we can verify the following result.
\begin{corollary}
Let $T\in \B(\Hc)$ and $S\in \B(\Hc)$ be generalized Drazin invertible which have closed range.  Then the following statements are equivalent:
\begin{enumerate}
\item [(i)] $S\leq ^ {GD-}T$.
    \item [(ii)] $SS^dS = SS^{GD-}T= TSS^d$.
    \item [(iii)] $SS^d =S^{GD-}T=TS^d$.
\end{enumerate}
\end{corollary}

Applying the core quasi-nilpotent decomposition, we state the following result in which we can classify all the operators ( say $X_j$ ) such that $T$ is below $X_j$ under the relation $\leq^{GD-}$.
\begin{theorem}
Let $T\in \B(\Hc)$ be generalized Drazin invertible which has closed range and consider the same decomposition as given in Theorem \ref{thm1}, for $T$ and $T^{-}$.  Then the following conditions are equivalent:
\begin{enumerate}[(i)]
\item $T \leq^{GD-}X$.
\item $X=\begin{bmatrix} T_1 ~&~ -T_1YX_4 \\ 0 ~&~ X_4 \end{bmatrix}~:~\begin{bmatrix} \N(P) \\ \R(P) \end{bmatrix} \to \begin{bmatrix} \N(P) \\ \R(P) \end{bmatrix}$, where $X_4\in \B(\R(P))$.
\end{enumerate}
\end{theorem}

\begin{theorem}\label{pod}
Let $S\in \B(\Hc)$ be generalized Drazin invertible which has closed range. Consider $\|S^d\|\leq 1$ and $\lim_{n\to \infty}(S^-S^{n+1}S^d-S^{n+1}S^dS^-)=0$. Then $S \leq^{GD-} T$ if and only if $S \leq^{d} T$.
\end{theorem}
\begin{proof}

Let $S \leq^{GD-}T$. From $SS^{GD-}=TS^{GD-}$, we obtain $SS^dSS^-=TS^dSS^-$ and consequently,
\begin{center}
  $SS^d=SS^dSS^-SS^d=TS^dSS^-SS^d=TS^d$.  
\end{center}
 Applying $S^dS=S^{GD-}S=S^{GD-}T=S^dSS^-T$, we get
\begin{eqnarray*}
 S^{d}S-S^{d}T &=&S^dSS^-T-S^dT=\left((S^d)^{n+2}S^{n+2}S^--(S^d)^{n+2}S^{n+1}\right)T\\
 &=& (S^d)^{n+1}\left(S^{n+2}S^dS^--S^{n+1}S^d\right)T.
\end{eqnarray*}
Using Theorem \ref{g1ep}, we have 
$\| S^{d}S-S^{d}T\| \leq \|(S^d)^{m+1}\| \|S^{n+2}S^dS^--S^{n+1}S^d\| \|T\|\to 0\mbox{ as $m \to \infty$}$. Hence $S \leq^ {d} T$.\\
Conversely, let $S \leq ^{d} T$. Then $S^{d}T=S^{d}S=TS^{d}$. Further,
\begin{center}
$SS^{GD-}=SS^dSS^{-}=TS^dSS^{-}=TS^{GD-}$, and
\end{center}
\begin{eqnarray*}
 S^{GD-}S-S^{GD-}T&=&SS^d-S^dSS^-T=S^dT-S^dSS^-T=\left((S^d)^{n+2}S^{n+1}-(S^d)^{n+2}S^{n+2}S^-\right)T\\
 &=& (S^d)^{n+1} \left(S^{n+1}S^d-S^{n+2}S^dS^-\right)T.
\end{eqnarray*}
Again applying Theorem \ref{g1ep}, we have 
\begin{center}
 $\|S^{GD-}S-S^{GD-}T\| \leq \|(S^d)^{n+1}\| \|S^{n+1}S^d-S^{n+2}S^dS^-\| \|T\|\to 0 \mbox{ as $m \to \infty$}$.    
\end{center}
Hence, $S \leq^{GD-} T$. 
\end{proof}
Notice that, in the proof of Theorem \ref{pod}, we have not used both limiting conditions from the Theorem \ref{g1ep}. Hence Theorem \ref{pod} is still true if we replace the limiting condition $\displaystyle\lim_{n\to \infty}(S^-S^{n+1}S^d-S^{n+1}S^dS^-)=0$ by $\displaystyle\lim_{n\to \infty}(S^{n+2}S^dS^--S^{n+1}S^d)=0$, which is restated in the below result.
\begin{corollary}
Let $S\in \B(\Hc)$ be generalized Drazin invertible which has closed range. Consider $\|S^d\|\leq 1$ and $\displaystyle\lim_{n\to \infty}(S^{n+2}S^dS^--S^{n+1}S^d)=0$. Then $S \leq^{GD-} T$ if and only if $S \leq^{d} T$.
\end{corollary}

\begin{corollary}
Let $S,T\in \mathcal{PO}$, where $\mathcal{ PO} = \left\{S\in \B(\Hc):  \|S^d\| \leq 1, \displaystyle\lim_{n \to \infty} \|  S^{n+2}S^dS^--S^{n+1}S^d  \| =0\right\}$. 
Then the relation 
$\leq ^ {GD-}$ is a pre-order on $\mathcal{ PO}$.
\end{corollary}

\section{1GD inverse}

In this section, we introduce 1GD inverse (called the dual of GD1) for bounded  Hilbert space operators. In addition, we state a few results in which the proofs are similar to those for GD1 inverses.
\begin{proposition}\label{prp-6}
Let $T\in \B(\Hc)$ be generalized Drazin invertible which has closed range.  Then $X=T^-TT^d$ is the unique solution of the following conditions:
\begin{equation*}
    TX=P_{\R(T^dT),\N(T^dT)} \mbox{ and } \R(X)\subset \R(T^-T).
\end{equation*}
\end{proposition}
In view of the Proposition \ref{prp-6}, we define the following representation of 1GD inverse for Hilbert space operators.
\begin{definition}
Let $T\in \B(\Hc)$ be generalized Drazin invertible which has closed range and $T^{-}$ be a fixed inner inverse of $T$. An operator $X$ is called the 1GD inverse of $T$ if it satisfies
\begin{equation*}
    TX=P_{\R(T^dT),\N(T^dT)} \mbox{ and } \R(X)\subset \R(T^-T).
\end{equation*}
\end{definition}
The 1GD inverse of $T$ is denoted by $T^{-GD}$ and it is explicitly represented by $T^{GD-}=T^{-}TT^d$.
\begin{theorem}
Consider $T\in \B(\Hc)$  and $T^{-}$ as defined in the Theorem \ref{thm1}. The GD1 inverse of $T$ is given by 
\begin{equation*}
    T^{-GD}=\begin{bmatrix} {T_1}^{-1} ~&~ 0 \\ Z ~&~ 0 \end{bmatrix}=T^d+\begin{bmatrix} 0 ~&~ 0 \\ Z ~&~ 0 \end{bmatrix}~:~\begin{bmatrix} \N(P) \\ \R(P) \end{bmatrix} \to \begin{bmatrix} \N(P) \\ \R(P) \end{bmatrix}.
    \end{equation*} 
Consequently, $T^{-GD}=T^d$ if and only if $Z=0$.
\end{theorem}
\begin{theorem}
Consider $T\in \B(\Hc)$  and $T^{-}$ as defined in the Theorem \ref{thm-2}. The 1GD inverse of $T$ is given by 
$$T^{-GD}=\begin{bmatrix} Z_1A_1A_1^d~ &~Z_1A_1^dA_2 \\ A_1^d~&~Z_3A_1^dA_2 \end{bmatrix}=Z_1A_1\begin{bmatrix} A_1^d~ &~(A_1^d)^2A_2 \\ 0~&~0 \end{bmatrix}+Z_3\begin{bmatrix} 0~ &~0 \\ A_1A_1^d~&~A_1^dA_2 \end{bmatrix}:~\begin{bmatrix} \R(T) \\ \N(T^*) \end{bmatrix} \to \begin{bmatrix} \R(T) \\ \N(T^*) \end{bmatrix}.$$
Consequently, $T^{-GD}=T^d$ if $Z_3=0$ and $A_1Z_1=Z_1A_1$.
\end{theorem}

\begin{theorem}
Let $T\in \B(\Hc)$ be generalized Drazin invertible which has closed range. Then 
\begin{enumerate}[(i)]
    \item $TT^{-GD}=T^{-GD}T$ if and only if $T^{-GD}=T^d$.
    \item $TT^{-GD}=T^{-GD}T$ if and only if $T^{-GD}=T^{GD-}=T^d$.
\end{enumerate}
 \end{theorem}

\begin{theorem}
Let $T\in \B(\Hc)$ be generalized Drazin invertible which has closed range. Then $X=T^-TT^{d}$ is the unique solution of operator equations: 
\begin{equation*}
XTX=X, ~~ XT^d =T^-T^d,~~\text{and}~~TX=TT^d.
\end{equation*}
\end{theorem}
\begin{definition}
Let $T\in \B(\Hc)$ and $S\in \B(\Hc)$ be generalized Drazin invertible which has closed range. We will say that  $S$ is below $T$  under the relation  $\leq ^ {-GD}$ if $SS^{-GD}= TS^{-GD}$ and $S^{-GD}S = S^{-GD}T$. We denote such relation by $S \leq ^{-GD}T$.
\end{definition}
\begin{proposition}
Let $T\in \B(\Hc)$ and $S\in \B(\Hc)$ be generalized Drazin invertible which has closed range. Then
\begin{enumerate}
    \item [(i)] $S^{-GD}S= S^{-GD}T$  if and only if $S^dS = S^dT$.
    \item [(ii)] $SS^{-GD} = TS^{-GD}$ if and only if $S^d = TS^{-}S^d$.
\end{enumerate}
\end{proposition}

\begin{corollary}
Let $T\in \B(\Hc)$ and $S\in \B(\Hc)$ be generalized Drazin invertible which has closed range.  Then the following statements are equivalent:
\begin{enumerate}
\item [(i)] $S\leq ^ {-GD}T$.
    \item [(ii)] $SS^dS = TS^{-GD}S= SS^dT$.
    \item [(iii)] $SS^d =TS^{-GD}= S^dT$.
\end{enumerate}
\end{corollary}
\begin{theorem}
Let $T\in \B(\Hc)$ be generalized Drazin invertible which has closed range and consider the same decomposition as given in Theorem \ref{thm1}, for $T$ and $T^{-}$.  Then the following conditions are equivalent:
\begin{enumerate}[(i)]
\item $T \leq^{-GD}X$.
\item $X=\begin{bmatrix} T_1 ~&~ 0 \\ -X_4ZT_1 ~&~ X_4 \end{bmatrix}~:~\begin{bmatrix} \N(P) \\ \R(P) \end{bmatrix} \to \begin{bmatrix} \N(P) \\ \R(P) \end{bmatrix}$, where $X_4\in \B(\R(P))$.
\end{enumerate}
\end{theorem}

\begin{theorem}
Let $T\in \B(\Hc)$ be generalized Drazin invertible which has closed range. Then $$\lim_{n\to \infty}(T^{n+1}T^dT^--T^-T^{n+1}T^d)=0$$ if and only if $\displaystyle\lim_{n\to \infty}(T^{n+2}T^dT^--T^{n+1}T^d)=0$ and $\displaystyle\lim_{n\to \infty}(T^-T^{n+2}T^d-T^{n+1}T^d)=0$.
\end{theorem}

\begin{theorem}
Let $S\in \B(\Hc)$ be generalized Drazin invertible which has closed range. Consider $\|S^d\|\leq 1$ and $\displaystyle\lim_{n\to \infty}(S^-S^{n+2}S^d-S^{n+1}S^d)=0$. Then $S \leq^{-GD} T$ if and only if $S \leq^{d} T$.
\end{theorem}

\begin{corollary}
Let $S,T\in \mathcal{PO}$, where $\mathcal{ PO} = \left\{S\in \B(\Hc):  \|S^d\| \leq 1, \displaystyle\lim_{n \to \infty} \|  S^-S^{n+2}S^d-S^{n+1}S^d  \| =0\right\}$. 
Then the relation 
$\leq ^ {-GD}$ is a pre-order on $\mathcal{ PO}$.
\end{corollary}

\section{Conclusion}
Our research introduces two new classes of inverses: GD1 and 1GD inverses for Hilbert space operators employing specific definitions. We have investigated some properties of these inverses along with its interconnection with the generalized Drazin inverse.  Further, some of the properties have also been investigated by considering the idempotent condition and decomposition's. Finally, GD1 inverses and 1GD inverses allow us to introduce binary relations. It will be helpful to mention a few key points for future work. 

\begin{enumerate}
\item[$\bullet$] Perturbation bounds related to the GD1 and 1GD inverses is an interesting for possible research.
\item[$\bullet$] GD1 and 1GD inverses of the sum of operators can be studied.
\item[$\bullet$] Investigation of the reverse order law for the class GD1 and 1 GD inverses, for Hilbert space operators would be an interesting idea for further research.
\item[$\bullet$] It is interesting to study the GD1 and 1GD inverses over the algebraic structure of a ring.
\end{enumerate}
\medskip

\section*{Conflict of interest}

None.

\section*{Data availability}
 
None.

\begin{acknowledgements}

\end{acknowledgements}

%
%

\bibliographystyle{abbrv}
\bibliography{Reference}   

\end{document}